\documentclass{article}
\usepackage{amsmath}
\usepackage{amsfonts}
\usepackage{amsthm}
\usepackage{amssymb}

\usepackage[active]{srcltx}

\newtheorem{theorem}{Theorem}[section]
\newtheorem{definition}[theorem]{Definition}
\newtheorem{lemma}[theorem]{Lemma}
\newtheorem{remark}[theorem]{Remark}
\newtheorem{problem}[theorem]{Problem}
\newtheorem{corollary}[theorem]{Corollary}

\newtheorem{proposition}[theorem]{Proposition}

\def\mb{\mathcal{B}}

\def\mx{\mathcal{X}}
\def\my{\mathcal{Y}}
\def\mz{\mathcal{Z}}

\newcommand{\conv}{{\rm conv}\hskip0.02cm}
\newcommand{\codim}{{\rm codim}\hskip0.02cm}

\newcommand{\dist}{{\rm dist}\hskip0.02cm}

\newcommand{\rk}{{\rm rank}\hskip0.02cm}

\newcommand{\ep}{\varepsilon}

\newcommand{\vora}{{\mathrm{VR}}}
\newcommand{\vv}{{\mathcal{V}}}
\newcommand{\R}{\mathbb R}
\newcommand{\N}{\mathbb N}

\newcommand{\affd}{\tilde{d}}
\newcommand{\ball}{\mathrm{B}}
\newcommand{\is}{\mathcal{S}}
\newcommand{\expect}{\mathbb{E}}
\newcommand{\prob}{\mathbb{P}}
\newcommand{\absconv}{\mathrm{absconv}}
\newcommand{\gau}{\mathbb{G}}
\newcommand{\vol}{\mathrm{vol}}

\def \dim{{\mathrm{dim}} \, }
\def \rank{{\mathrm{rank}} \, }
\def \ker{{\mathrm{ker}} \, }

\renewcommand{\span}{\mathrm{span}}
\def \vr{\varepsilon}

\def \ran{{\mathrm{ran}} \, }
\def \codim{{\mathrm{codim}} \,}
\def \diag{{\mathrm{diag}} \, }

\def \eqalign#1{\null\,\vcenter{\openup\jot 
   \ialign{\strut\hfil$\displaystyle{##}$&$
      \displaystyle{{}##}$\hfil \crcr#1\crcr}}\,}
\begin{document}

\title{\LARGE{Dependence of Kolmogorov widths on the ambient space}}

\author{T.~Oikhberg\\
Department of Mathematics\\
University of California-Irvine
Irvine, CA, 92697\\
and\\
Department of Mathematics\\
University of Illinois at Urbana-Champaign\\
Urbana, IL 61801\\
e-mail: {\tt toikhber@math.uci.edu}
\and\\
M.\,I.~Ostrovskii\\
Department of Mathematics and Computer Science\\
St. John's University\\
8000 Utopia Parkway\\
Queens, NY 11439\\
USA\\
e-mail: {\tt ostrovsm@stjohns.edu}}

\date{\today}
\maketitle

\begin{abstract}
We study the dependence of the Kolmogorov widths of a compact set
on the ambient Banach space.
\end{abstract}

\begin{large}
\tableofcontents

\section{Introduction}

Let $\mz $ be a subset of a Banach space $\mx $ and $x\in \mx $.
The {\it distance from $x$ to $\mz $} is defined as
\[E(x,\mz )=\inf\{||x-z||:~z\in \mz \}.\]
\begin{definition}
{\rm Let $K$ be a subset of a Banach space $\mx $, $n\in\mathbb{N}\cup\{0\}$.
The {\it Kolmogorov $n$-width} (or {\it $n$-th Kolmogorov number})
of $K$ is given by \[d_n(K, \mx)=\inf_{\mx _n}\sup_{x\in K}E(x,\mx _n),\]
where the infimum is over all subspaces $\mx_n \subset \mx$, of dimension
not exceeding $n$. We use the notation $d_n(K)$
if $\mx$ is clear from context.}
\end{definition}

This notion was introduced by Kolmogorov \cite{Kol36} in 1936. It
has been a subject of an extensive study and has found many
applications, both in Approximation Theory and in Functional
Analysis, see \cite{CS}, \cite{LGM96}, \cite{Pie80}, \cite{Pin85},
and \cite{Tik60}. In \cite{OS09} it was discovered that some
general asymptotic properties of Kolmogorov widths are useful in
the study of closures of sets of operators in the weak operator
topology. More results on asymptotic properties of Kolmogorov
widths were discovered in \cite{Ost10}. The purpose of this paper
is to continue analysis of asymptotic properties of widths.
\medskip

Our emphasis in this paper is on dependence of asymptotic
properties of widths on the ambient space. It is known for long
time (see \cite[\S 7]{Tik60}) that if $\my$ is a subspace of a
Banach space $\mx$ and $K\subset\my$, then it can happen that
$d_n(K,\my)>d_n(K,\mx)$. Furthermore, the quotient
$d_n(K,\my)/d_n(K,\mx)$ can be arbitrarily large. An example with
in a certain sense optimal order of this quotient was found in
\cite{Ost10}, where the following result was proved:

\begin{theorem}[\cite{Ost10}]\label{T:DepSpace}
For each $n$ the Banach space $\ell_1^{3n}$ contains a $2n$-dimensional subspace $Y_{2n}$ and
a compact $K_{2n}\subset \my_{2n}$ such that
$d_n(K_{2n},\ell_1^{3n})\le 1$ but $d_n(K_{2n},\my_{2n})\ge
c\sqrt{n}$ for some absolute constant $c>0$.
\end{theorem}

\begin{remark}\label{R:Optimal} The order in Theorem \ref{T:DepSpace} is optimal in the following sense:
Proposition \ref{P:nDim} implies that $d_n(K_{2n},\my_{2n})\le
\sqrt{2n}\,d_n(K_{2n},\ell_1^{3n})$.
\end{remark}

The paper is structured as follows: in Section \ref{s:abs wi}, we
introduce the notion of the absolute width $d_n^a(K)$ (Definition
\ref{D:AW}), and collect the necessary basic facts. In general,
$d_n^a(K) \leq d_n(K)$, but in some cases, we obtain the equality,
or at least proportionality, of the two quantities. In Section
\ref{s:affine}, we study affine widths. This allows us to
construct, in certain Banach spaces $X$, a compact convex set $K$
so that $d_1(K) > d_1^a(K)$. In Section \ref{s:compare} we note
some connections of Kolmogorov and absolute widths to other
$s$-sequences (such as the sequences of Gelfand numbers). This
provides us with some tools to be used later.\medskip

We then pass to the study of asymptotic behavior of Kolmogorov
numbers. In Section \ref{s:vr}, we exhibit a large class of Banach
spaces which contain a sequence of compact subsets $(K_n)$, so
that $\lim_n d_{k_n}(K_n)/d_{k_n}^a(K_n) = \infty$, for some
increasing sequence $(k_n)$. In Section \ref{s:ratio}, we sharpen
this result by showing that, if a space $\mx$ satisfies certain
conditions (for instance, if it is $K$-convex), then it contains a
compact $K$ with the property that $\limsup_n d_n(K)/d_n^a(K) =
\infty$. If, furthermore, $\mx$ contains $\ell_p$ ($1 < p <
\infty$) as a complemented subspace, then it contains a compact
subset $K$ so that $\liminf_n n^{-\sigma} d_n(K)/d_n^a(K) =
\infty$, for some $\sigma > 0$. In Section \ref{s:restrict}, we
examine compacts $K$ for which $d_n(K) = d_n^a(K)$, for any
ambient space. Finally, Section \ref{s:image} is devoted to
comparing the Kolmogorov widths of the sets $K$ and $u(K)$, where
$u$ is compact operator.\medskip

Throughout the paper we pose some interesting geometric
problems related to our study (Problems \ref{P:Nspace},
\ref{P:Linfty}, \ref{P:ProjRel}, \ref{P:ratio}, \ref{P:to infty},
\ref{P:RevKKM}, \ref{P:omalComp}). Problem \ref{P:ProjRel} could
be of interest not only in the context of the theory of
widths.\medskip

We use the basic Banach space theory and its standard notation. We
denote by $\ball(\mx)$ the closed unit ball of a space $\mx$.

\section{Absolute widths}\label{s:abs wi}

Dependence of the sequence $\{d_n(K)\}_{n=0}^\infty$ on the
ambient Banach space leads to the introduction of the following
definition.

\begin{definition}[\cite{Ism74}]\label{D:AW}
{\rm Let $K$ be a compact in a Banach space $\my$ and
$n\in\mathbb{N}$. The $n$-th {\it absolute width} (or {\it number}) $d^a_n(K)$ of
$K$ is defined by $d^a_n(K)=\inf_\mx d_n(K,\mx)$, where the $\inf$
is over all Banach spaces $\mx$ containing $\my$ as a subspace.}
\end{definition}

Absolute widths were studied in \cite{Ism74}, \cite{Koc90},
\cite{Oik95}, and \cite{Ost10}. Our main purpose in this paper is
to study the asymptotic behavior of the quotients
$d_n(K,\my)/d_n^a(K)$ under different assumptions. We start with
the following natural open problem: characterize Banach spaces $\my$ for which
$d_n(K,\my)=d_n^a(K)$ for all compacts $K\subset \my$.

We present a class of Banach spaces having this property. The
following definition goes back to \cite{LP68}: Let $1\le
\lambda<\infty$. A Banach space $\my$ is called an {\it
$\mathcal{L}_{\infty,\lambda}$-space} if for every
finite-dimensional subspace $S\subset \my$ there is a
finite-dimensional subspace $F\subset \my$ such that $S\subset F$
and $d(F,\ell_\infty^m)\le\lambda$, where $m=\dim F$. A Banach
space is called an {\it $\mathcal{L}_{\infty,\lambda+}$-space} if
it is a $\mathcal{L}_{\infty,\nu}$-space for each $\nu>\lambda$.
See \cite{Bou81} and \cite{LT73} for theory of
$\mathcal{L}_{p}$-spaces.

More generally, a Banach space $\mx$ is
called an {\it ${\mathcal{N}}_\lambda$-space} if, for every finite
dimensional subspace $E$ of $X$, there exists a finite dimensional
subspace $F$, satisfying $E \subset F \subset X$ and
$\lambda(F)\le\lambda$. Here, following \cite{Tom89}, we define
$\lambda(F)$ the ({\it absolute}) {\it projection constant} of $F$
as follows: for a superspace $G \supset F$, define the
{\it relative projection constant} $\lambda(F,G)$ as the infimum
of $\|P\|$, where $P$ is the projection from $G$ onto $F$. Then
$\lambda(F) = \sup \lambda(F,G)$, with the supremum taken over
all superspaces $G$.

A Banach space $X$ is
called an {\it ${\mathcal{N}}_{\lambda+}$-space} if it is a
${\mathcal{N}}_{\nu}$-space for each $\nu>\lambda$, and an {\it
${\mathcal{N}}$-space} if it is a ${\mathcal{N}}_\lambda$-space
for some $1\le\lambda<\infty$.

It is easy to see that each
$\mathcal{L}_{\infty,\lambda}$-space is an
$\mathcal{N}_\lambda$-space.
However, the converse is false, see e.g. \cite{Sza90}.
It is not known whether each $\mathcal{N}$-space is an
$\mathcal{L}_{\infty,\lambda}$-space for some $\lambda<\infty$.
This problem is a version of the well-known $P_\lambda$-problem
(see \cite[Problem 7, p.~323]{LP68}), which is still open.
However, it is known \cite{LL66} that, for a real Banach space $\mx$,
the following are equivalent: (i) $\mx$ is a $\mathcal{N}_{1+}$-space;
(i) $\mx$ is a $\mathcal{L}_{\infty, 1+}$-space; (iii) $\mx^* = L_1(\mu)$,
for some measure $\mu$.

\begin{proposition}\label{P:LinfAbs}
Let $K$ be a compact in an $\mathcal{N}_{\infty,\lambda+}$-space $\my$.
Then $d_n(K,\my)\le\lambda d_n^a(K)$ for all $n\in\mathbb{N}$.
\end{proposition}

\begin{proof}
It suffices to show that for each $C>\lambda$ and $n \in \N$ we have
$d_n(K,\my)\le C d_n^a(K)$. Pick $\ep>0$ so that $(1+3 \vr + \vr^2)\lambda < C$.
By the definition of $d_n^a$ there exists a Banach space
$\mx\supset \my$ and an $n$-dimensional subspace $\mx_n\subset\mx$
such that $E(x,\mx_n)\le (1+\ep)d_n^a(K)$ for any $x\in K$. Let
$\{k_i\}\subset K$ be an $\ep\lambda d_n^a(K)$-net in $K$.
Find a finite dimensional subspace $F \subset \my$, containing
$\{k_i\}$, so that there exists a projection $P:\mx\to F$ satisfying
$\|P\|\le \lambda(1+\ep)$. Let $\my_n=P(\mx_n)$. Then
$E(k_i,\my_n)=E(Pk_i,P\mx_n)\le (1+\ep)\lambda E(k_i,\mx_n)\le
(1+\ep)^2\lambda d_n^a(K)$. Let $k\in K$ and $k_i$ be such that
$||k-k_i||\le \ep\lambda d_n^a(K)$, we have
$$
E(k,\my_n)\le
||k-k_i||+E(k_i,\my_n)\le ((1+\ep)^2+\ep)\lambda d_n^a(K) \leq
C d_n^a(K) .
\, \, \, \qedhere
$$
\end{proof}

\begin{corollary}\label{C:LinfAbs}
Let $K$ be a compact in an $\mathcal{L}_{\infty,1+}$-space $\my$.
Then $d_n(K,\my)=d_n^a(K)$ for all $n\in\mathbb{N}$.
\end{corollary}

In this connection it is worth mentioning that all spaces of
continuous functions on compacts with their $\sup$-norms are
$\mathcal{L}_{\infty,1+}$-spaces, see \cite{LT73}.

\begin{remark} Corollary \ref{C:LinfAbs} can be regarded as a generalization of the following result of
Ismagilov \cite[Corollary of Theorem 2]{Ism74}: Let $K$ be a
compact in a Banach space $\mx$ and $\mb$ be the Banach space of all
bounded functions on $\ball(\mx^*)$ (the unit ball of $\mx^*$) with the
$\sup$-norm. Let $i$ be the natural isometric embedding of $\mx$
into $\mb$. Then $d_n^a(K)=d(i(K),\mb)$. To get this result from
Corollary \ref{C:LinfAbs} it suffices to combine the corollary
with the well-known fact that $\mb$ is an $\mathcal{L}_{\infty,
1+}$-space (see \cite{LT73}).
\end{remark}

Do Proposition \ref{P:LinfAbs} and Corollary \ref{C:LinfAbs} characterize
the ${\mathcal{N}}$ spaces and ${\mathcal{L}}_{\infty,1+}$ spaces, respectively?

\begin{problem}\label{P:Nspace}
Let a Banach space $\my$ be such that for some
$1\le\lambda<\infty$ the condition $d_n(K,\my)\le\lambda d_n^a(K)$
holds for each compact $K\subset \my$ and each $n\in\mathbb{N}$.
Does it follow that $\my$ is an $\mathcal{N}$-space?
\end{problem}

\begin{problem}\label{P:Linfty} Let
a Banach space $\my$ be such that $d_n^a(K)=d_n(K,\my)$ for each
compact $K\subset \my$ and each $n\in\mathbb{N}$. Does it follow
that $\my$ is an $\mathcal{L}_{\infty,1+}$-space?
\end{problem}

Approaches to these questions may rely on
Zippin's solution \cite{Zip81a,Zip81b,Zip84} to the
close-to-isometric version of the $P_\lambda$-problem. (See
\cite{Tom89} for a presentation of this result of Zippin and
\cite{Zip00} for further results related to the
$P_\lambda$-problem.)

Corollary \ref{C:LinfAbs} can be used to estimate from above the
quotient $d_k(K)/d_k^a(K)$ for an $n$-dimensional compact $K$.

\begin{proposition}\label{P:nDim} Let $K$ be an $n$-dimensional compact in a Banach space $\my$.
Then $d_k(K,\my)\le\sqrt{n}d_k^a(K)$ for all $k\in\mathbb{N}$.
\end{proposition}

\begin{proof} We may assume that $\my$ is separable and so we may consider
$\my$ as a subspace of $\ell_\infty(I)$.
It is easy to see that $\ell_\infty(I)$ is an
$\mathcal{L}_{\infty,1+}$-space. By Corollary \ref{C:LinfAbs},
$d_n^a(K) = d_n(K, \ell_\infty(I))$.
\medskip

The inequality $d_k(K,\my)\le\sqrt{n}d_k^a(K)$ is trivially true
for $k\ge n$. So let $k\in\{0,\dots,n-1\}$. Consider an arbitrary
$\ep>0$. Let $\mx_k$ be a $k$-dimensional subspace of
$\ell_\infty$ such that $E(x,\mx_k)\le (1+\ep)d_k^a(K)$ for all
$x\in K$. Let $P:\ell_\infty(I) \to\span[K]$ be a linear projection
with norm $\le\sqrt{n}$, existing by the Kadets-Snobar theorem
\cite{KS71} and let $\my_k=P\mx_k$. Then for all $x\in K$ we have
$E(x,\my_k)=E(Px,P\mx_k)\le ||P||E(x,\mx_k)\le
\sqrt{n}(1+\ep)d_k^a(K)$.
\end{proof}

As we already mentioned in Remark \ref{R:Optimal}, the estimate of
Proposition \ref{P:nDim} is optimal up to a multiplicative
constant.\medskip

As a step towards the solution of Problems \ref{P:Linfty} and
\ref{P:Nspace} we find a wide class of spaces $X$ for which the
quotients $d_n(K,X)/d_n^a(K)$ can be arbitrarily large. This is
the subject of Sections \ref{s:vr} and \ref{s:ratio}.

\section{Affine widths, geometry, and injectivity}\label{s:affine}

While dealing with arbitrary convex (not necessarily centrally symmetric) sets,
it is convenient to use affine subspaces for approximation (see e.g.~\cite{AO}).
\begin{definition}\label{D:aff_widths}
{\rm Let $K$ be a compact in a Banach space $Y$ and $n\in\N \cup
\{0\}$. The $n$-th {\it affine width} $\affd_n(K)$ of $K$ is set
to be $\inf_Z \sup_{x \in K} E(x,Z)$, where the infimum runs over
all affine subspaces of $Z \subset Y$ of dimension not exceeding
$n$. The $n$-th {\it absolute affine width} $\affd^a_n(K)$ of $K$
is defined by $\affd^a_n(K)=\inf_X \affd_n(K,X)$, where the $\inf$
is over all Banach spaces $X$ containing $Y$ as a subspace.}
\end{definition}
It is clear that $\affd^a_n(K) \leq \affd_n(K,X)$, and the equality
is attained if $X$ is $1$-injective. Moreover (see \cite[Section 6.2]{AO}),
$$
d_n(K) \geq \affd_n(K) \geq d_{n+1}(K \cup (-K)) .
$$
Furthermore, $d_n(K) = \affd_n(K)$ if $K$ is centrally symmetric.
The affine widths $\affd_0$ have been considered previously.
To summarize them, recall a few definitions.
\begin{definition}\label{D:Jung}
{\rm
For a bounded subset $K$ of a Banach space $\my$, define its {\it diameter}
$D(K)$ and {\it radius} $R(K)$ by setting
$$
D(K) = \sup_{a,b \in K} \|a-b\| , \, \, \,
R(K) = \inf_{y \in \my} \sup_{a \in K} \|a-y\|
$$
(that is, $R(K)$ is the infimum of the radii of balls containing $K$).
The {\it Jung constant} $J(\my)$ of a Banach space $\my$ is defined as the
supremum (over bounded sets $K \subset \my$) of $2 R(K)/D(K)$.
Note that, in our notation, $R(K) = \tilde{d}_0(K)$
}
\end{definition}

Clearly, $2 \geq J(\my) \geq 1$.
The spaces $\my$ with $J(\my) = 1$ were described in \cite{Dav}.

\begin{theorem}[\cite{Dav}]\label{T:injective}
For a real Banach space $\my$, the following are equivalent:
\begin{enumerate}
\item
For any compact $K \subset \my$, there exists $y \in \my$ such that
$K \subset \ball(y, D(K)/2)$.
\item
$Y$ is $1$-injective.
\item
$J(\my) = 1$.
\end{enumerate}
\end{theorem}

The equivalence (1) $\Leftrightarrow$ (2) in the above theorem precedes
\cite{Dav} -- it is due to \cite{Nach}.
For certain Banach spaces, the Jung constant is known. For instance,
\cite{Ball, Pi} show that, for $1 \leq p < \infty$,
$J(L_p(\mu)) = \max\{2^{1/p}, 2^{(p-1)/p}\}$.
By \cite{FS}, for any rearrangement invariant space $\my$ which
is not injective, $J(\my) \geq \sqrt{2}$, and the equality holds iff
$\my$ is isometric to the Hilbert space. \cite{AFS} establishes the
Jung constant for some classes of Banach lattices (such as Lorentz spaces).
One is referred to the bibliography of the latter paper for additional information.
In our notation, Theorem~\ref{T:injective} implies that, for
any bounded $K$ is a $1$-injective Banach space $\my$, $\affd_0^a(K) = D(K)/2$.
For any Banach space $\my$,
$J(\my) = \sup_{K \subset \my \, {\mathrm{bounded}}} \affd_0(K)/\affd_0^a(K)$.
This leads to:

\begin{proposition}\label{P:d_1}
Suppose a real Banach space $\mx$ is not $1$-injective. Then
$\tilde{\mx} = \R \oplus_1 \mx$ contains a bounded centrally symmetric subset $K$,
such that $d_1^a(K) < d_1(K)$.
\end{proposition}


\begin{proof}
By Theorem~\ref{T:injective}, $\mx$ contains a bounded set $A$, such that
$D(A) = 1/2$, while $R(A) = c \in (1/4, 1/2]$. By translation, we may assume that
$\|x\| \leq 1/2$ for any $x \in A$. Consider the ``skew cylinder''
$$
K = \conv \big( 1 \oplus A, (-1) \oplus (-A) \big) =
\Big\{ t \oplus \Big( \frac{1+t}{2} a_1 - \frac{1-t}{2} a_2 \Big)
: -1 \leq t \leq 1 , \, a_1, a_2 \in A \Big\} .
$$
We shall show that $d_1(K) \geq c$, while $d_1^a(K) \leq 1/4$
(in fact, equalities hold in both cases, but we do not need this for
our purposes).
We handle $d_1^a(K)$ first. Embed $\mx$ into a $1$-injective space $\tilde{\mx}$.
By the discussion above, there exists $\tilde{x} \in \tilde{\mx}$ such that
$\|\tilde{x} - a\| \leq 1/4$ for any $a \in A$. Consider the $1$-dimensional
space $F = \span [1 \oplus \tilde{x}] \subset \R \oplus_1 \tilde{\mx}$, and show that,
for any $y \in K$, $E(y,F) \leq 1/4$. Indeed, write $y = t \oplus a$, where
$t \in [-1,1]$, and
$$
a = \frac{1+t}{2} a_1 - \frac{1-t}{2} a_2  \, \, ( a_1, a_2 \in A ) .
$$
Then $t \oplus t \tilde{x} \in F$, hence
$$
E(y,F) \leq \|y - t \oplus t \tilde{x}\| = \|a - t \tilde{x}\| =
\Big\| \frac{1+t}{2} (a_1 - \tilde{x}) - \frac{1-t}{2} (a_2 \tilde{x}) \Big\| \leq
\frac{1}{4} \Big( \frac{1+t}{2} + \frac{1-t}{2} \Big) = \frac{1}{4} .
$$

Turning to $d_1(K)$, we have to show that, for any $1$-dimensional
subspace $F$ of $\R \oplus \mx$, we have $\sup_{a \in A} E(1
\oplus a, F) \geq c$. If $F = \span[0 \oplus x] \subset \R \oplus_1 \mx$,
the previous inequality holds for every $a$. Now consider
$F = \span[1 \oplus x] \subset \R \oplus_1 \mx$.
Note that, for $a \in A$,
$E(1 \oplus a, F) = \inf_{t \in \R} (|1-t| + \|tx - a\|)$.
Consider the cases of $\|x\| \leq 1$ and
$\|x\| > 1$ separately. (i) If $\|x\| \leq 1$,
$$
|1-t| + \|tx - a\| = |1-t| + \|(x - a) - (1-t)x\| \geq
|1-t| + \|x - a\| - |1-t| \|x\| \geq \|x - a\| ,
$$
hence $\sup_{a \in A} E(1 \oplus a, F) \geq \sup_{a \in A} \|x - a\| \geq c$.
(ii) If $\|x\| > 1$,
$$
|1-t| + \|tx - a\| \geq 1 - |t| + |t| \|x\| - \|a\| \geq 1 - \|a\| \geq \frac{1}{2} .
$$
As $c \leq 1/2$, we are done.
\end{proof}

We obtain a sharper result for $\mx = L_1(\mu)$.

\begin{proposition}\label{P:l1}
Suppose the real Banach space $L_1(\mu)$ ($\mu$ is a $\sigma$-finite measure)
has dimension at least $n = 2^k+1$ ($k \geq 2$).
Then $L_1(\mu)$ contains a closed finite dimensional
centrally symmetric subset $K$, satisfying $d_1^a(K) \leq 1/4$, and
$d_1(K) \geq (n-1)/(2n)$.
\end{proposition}

This result is asymptotically optimal: by Proposition~\ref{k-th},
$d_1(K) \leq 2 d_1^a(K)$.

\begin{proof}
By assumption, $L_1(\mu)$ contains a contractively complemented copy of $\ell_1^n$.
Thus, it suffices to prove the existence of a set $K \subset \ell_1^n$ with
desired properties. Write $\ell_1^n = \R \oplus_1 \ell_1^{n-1}$.
By \cite{Dol}, $J(\ell_1^{n-1}) = 2(n-1)/n$. By the compactness of the set of bounded
compacts in a finite dimensional space (with respect to the Hausdorff distance),
$\ell_1^{n-1}$ contains a set $A$ with diameter $1/2$, and radius $(n-1)/(2n)$.
We construct $K$ as in the proof of Proposition~\ref{P:d_1}.
\end{proof}
\begin{remark}\label{hadamard}
In fact, \cite{Dol} shows that $J(\ell_1^{n-1}) = 2(n-1)/n$ iff there exists
a Hadamard matrix of order $n$. Walsh matrices are clearly Hadamard matrices of order $2^k$.
The existence of Hadamard matrices of order $4k$ for any $k \in \N$ is a long-standing
conjecture.
\end{remark}

\section{Relations with other sequences of $s$-numbers}\label{s:compare}

In this section, we consider the relations between Kolmogorov and
absolute numbers of operators, on one hand, and other sequences of
$s$-numbers, on the other hand. For general properties of $s$-numbers
(or $s$-sequences), we refer to \cite{Pie87}. We define the {\it Kolmogorov}
and {\it absolute} {\it widths} ({\it numbers}) of an operator
$T \in B(\mx,\my)$ by setting $d_n(T) = d_n(\overline{T(\ball(\mx))}$, and
$d_n^a(T) = d_n^a(\overline{T(\ball(\mx))}$. We also need to define the
{\it approximation} and {\it Gelfand} {\it numbers} of $T$, denoted
by $c_n$ and $a_n$, respectively:
$$
\begin{array}{lll}
a_n(T)
& = &
\inf \{ \|T-S\| : S \in B(\mx,\my), \, \rank S \leq n \} ,
\cr
c_n(T)
& = &
\inf \{ \|T|_E\| : E \subset \mx, \, \codim E \leq n \} .
\cr
\end{array}
$$

Note that $d_n(T)\le a_n(T)$, $c_n(T)\le a_n(T)$, and $d_n(T) =
\inf \|q T\|$, where the infimum runs over all quotient maps $q :
\my \to \my/F$, with $\dim F \leq n$.

By \cite{Pie87}, $s$-numbers (such as $a_n( \cdot )$, $c_n( \cdot
)$, and $d_n( \cdot )$) have an ideal property: \[s_n(ATB) \leq
\|A\| s_n(T) \|B\|\] for any three operators $A$, $B$, and $T$.

The following lemma seems to be part of the Banach space lore.
\begin{proposition}\label{appr}
Consider an operator $T \in B(\mx,\my)$, and $n \in \N$.
\begin{enumerate}
\item If $\my$ is $\lambda$-injective, then $a_n(T) \leq \lambda
c_n(T)$. \item If $\mx$ is $\lambda$-projective, then $a_n(T) \leq
\lambda d_n(T)$.
\end{enumerate}
\end{proposition}
\begin{proof}
We only prove (2). Suppose $d_n(T) < 1$, and show that there
exists an operator $u : \mx \to \my$, of rank $\le n$, with $\|T -
u\| < \lambda$. To this end, pick a subspace $F \subset \my$, such
that $\dim F \le n$, and $\|q_F T\| < 1$ (here, $q_F : \my \to
\my/F$ is the quotient map). As $\mx$ is $\lambda$-projective,
$qT$ admits a lifting $T_0 : \mx \to \my$, with $\|T_0\| <
\lambda$ and $q T_0 = qT$. Let $u = T - T_0$. As $q u = 0$, the
range of $u$ must be contained in $F$, hence $\rank u \leq \dim F
\le n$.
\end{proof}

In a similar fashion, one can show:

\begin{proposition}\label{equal}
Consider $T \in B(\mx,\my)$, and $n \in \N$.
\begin{enumerate}
\item If $\mx$ is $1$-projective, then $d^a_n(T) = c_n(T)$. \item
If $\my$ is $1$-injective, then $d^a_n(T) = d_n(T)$.
\end{enumerate}
\end{proposition}
\begin{proof}
Here, we prove (1). Let $J$ be an embedding of $\my$ into a
$1$-injective space $\my_0$. By Proposition~\ref{appr}(2), $d^a_n(T)
= d_n(JT) = a_n(JT) \geq c_n(JT) = c_n(T)$. Conversely, by
Proposition~\ref{appr}(1), $a_n(JT) \leq c_n(JT)$.
\end{proof}
\begin{proposition}\label{k-th}
For any $T \in B(\mx,\my)$ and $k \in \N$, $d_k(T) \leq \sqrt{2(k+1)}
\, d^a_k(T)$.
\end{proposition}
\begin{proof}
Fix a quotient map $Q : X_0 \to X$, where $X_0$ is $1$-projective.
Clearly, $d_k(T) = d_k(TQ) \leq a_k(TQ)$, and $d^a_k(T) =
d^a_k(TQ)$. By Proposition~\ref{equal}, $d^a_k(T) = c_k(TQ)$. By
\cite[Proposition 2.4.3]{CS}, $a_k(TQ)\leq \sqrt{2(k+1)} \,
c_k(TQ)$.
\end{proof}

\begin{lemma}\label{L:AbsGel} For any operator $u$, $c_n(u)\ge d_n^a(u)$.
\end{lemma}

Some cases of equality are noted in Propositions \ref{appr} and \ref{equal}.

\begin{proof} For $u \in B(\mx,\my)$, consider an
isometric embedding $j$ of $\my$ into $\ell_\infty(I)$, for
a sufficiently large index set $I$. Let $E\subset \mx$
be a subspace of codimension $n$ on which $\|u|_E\|<\lambda$.
We need to show that $d_n^a(u(\ball(\mx)))<\lambda$.
It suffices to show that $d_n(ju(\ball(\mx)))<\lambda$. 
Using the injectivity of $\ell_\infty(I)$, we obtain
$\tilde{v} \in B(\mx, \ell_\infty(I))$ so that
$\tilde{v}|_E = ju|_E$, and $\|\tilde{v}\| = \|ju|_E\| < \lambda$.
Let $w=\tilde v-ju$. Then $\|ju+w\|<\lambda$ and $\rk w\leq n$.
This implies that $d_n(ju(\ball(\mx))) \leq E(ju(\ball(\mx)), w(\mx)))<\lambda$.
\end{proof}

Finally, we state a well known result, to be used throughout the paper.

\begin{lemma}\label{L:ops}
Suppose $K$ is a subset of a Banach space $\mx$, and $T \in B(\mx,\my)$.
Then, for any $n \in \N$, $d_n(T(K),\my) \leq \|T\| d_n(K,\mx)$, and
$d_n^a(T(K)) \leq \|T\| d_n(K)$.
\end{lemma}

\begin{proof}[Sketch of the proof]
(i) For any $C > d_n(K,\mx)$, there exists $F \subset \mx$, so that
$\dim F \leq n$, and $E(K,F) < C$. Then
$d_n(T(K),\my) \leq E(T(K), T(F)) < C \|T\|$.
Taking the infimum over all $C$'s, we conclude that
$d_n(T(K),\my) \leq \|T\| d_n(K,\mx)$.

(ii) Embed $\mx$ and $\my$ isometrically into $\ell_\infty(I)$
and $\ell_\infty(J)$, respectively. Then $T$ has an extension
$S : \ell_\infty(I) \to \ell_\infty(J)$, with $\|T\| = \|S\|$.
We know that $d_n^a(K) = d_n(K, \ell_\infty(I))$, and
$d_n^a(T(K)) = d_n(S(K), \ell_\infty(J))$. By Part (i),
$d_n(S(K), \ell_\infty(J)) \leq \|S\| d_n(K, \ell_\infty(I))$.
\end{proof}

\section{A class of spaces for which the ratio between widths and
 absolute widths can be arbitrarily large}\label{s:vr}

Throughout this section,
$\ball_p^m$ stands for the unit ball of $\ell_p^m$.
We use $\vora(F)$ to denote the volume ratio of a
finite-dimensional normed space $F$, that is
$\vora(F)=\vol(\ball(F))/\vol(\mathcal{E})$, where $\mathcal{E}$
is the maximum volume ellipsoid in $\ball(F)$, see \cite{ST80} or
\cite{Pis89} for basic facts about $\vora$. The purpose of this
section is to prove the following result.

\begin{theorem}\label{thm:unif_bded_VR}
Let $\mx$ be a Banach space containing a sequence $\{\mx_n\}$ of
uniformly complemented subspaces with $\dim \mx_n\to\infty$ and such
that there exists $\gamma \in [0,1/2)$ satisfying
$$
\liminf_{n \to \infty} \frac{\vora(X_n)}{(\dim \mx_n)^{\gamma}} = 0 .
$$
Then there exist a sequence of compacts $K_n\subset \mx$ with
\[\lim_{n\to\infty}\frac{d_n^a(K_n)}{d_n(K_n,\mx)}=0.\]
\end{theorem}

The proof relies on the following finite dimensional theorem.

\begin{theorem}\label{thm:vol_rat}
Suppose $\gamma \in [0,1/2)$ and $\sigma \in (\gamma,1/2)$. Let $A
\geq 5$ be a positive integer satisfying
$$
\frac{A-2}{2(A+1)} \geq \gamma \frac{A}{A-1} + (\sigma - \gamma) .
$$
Then there exists $N_0 \in \N$ with the following property:
if $n \geq N_0$ is even, and $X$ is a normed space of dimension $An$,
with $\vora(X) \leq n^\gamma$, then there exists a compact symmetric
$K \subset X$, so that $d_n^a(K) \leq C_1$, and $d_n(K,X)
\geq n^{\sigma - \gamma}$, where $C_1$ is a constant which depends
only on $A$.
\end{theorem}

Note that, for $\gamma$ and $\sigma$ as above, $A$ satisfying
the centered identity always exists. Indeed, as $A \to \infty$,
the left hand side tends to $1/2$, and the right hand side --
to $\sigma < 1/2$.

Tools which we use in this proof were invented by Gluskin
\cite{Glu81} and later developed by Szarek \cite{Sza81} and
\cite{Sza86}. See \cite{MT03} for a survey of related results.
Throughout the proof we use Gaussian random variables. To describe
them, denote an orthonormal basis in $\R^N$ by $(e_i)$. We call a
vector $\sum_{i=1}^N {g}_i e_i$ \emph{$N$-standard
Gaussian} if ${g}_i$ are independent standard normal
random variables (with $\expect(|{g}_i|^2) = 1)$. It is
well known that the definition is actually independent of the
choice of an orthonormal basis in $\R^N$. If $P$ is an orthogonal
projection on an $M$-dimensional subspace of $\R^N$, and
$(\tilde{g}_j)_{j=1}^k$ are independent $N$-standard Gaussians,
then $(P \tilde{g}_j)_{j=1}^k$ are independent $M$-standard
Gaussians (see e.g.~\cite[Fact 1]{MT03}).

Proving Theorem \ref{thm:vol_rat} we identify $X$ with $\R^{An}$,
and naturally embed it into $\tilde{X} = \R^{(1+A)n}$, with the
basis $(e_i)_{i=1}^{(1+A)n}$. We may and shall assume that the
maximal volume ellipsoid, inscribed in $\ball(X)$, is the
Euclidean ball $\ball^{An}_2$. Let $P_X$ be the orthogonal
projection of $\tilde{X}$ onto $X$. Let $\tilde{g}_i =
\tilde{g}_{i,\omega}$ ($1 \leq i \leq (1+A)n$, $\omega \in
\Omega$) be independent $(1+A)n$-standard Gaussian vectors in
$\tilde{X}$. Then $g_i = g_{i,\omega} = P_X \tilde{g}_i$ are
$An$-standard Gaussian vectors in $X$. We show that the set $K =
K_\omega = \absconv(g_1, \ldots, g_{(1+A)n})$ has the desired
properties with probability (relative to $\omega$) of at least
$1/2$, for sufficiently large $n$.
We use the notation $\gau = \gau_\omega =
(\tilde{g}_{i,\omega})_{i=1}^{(1+A)n}$.
Let $\tilde{K} = \tilde{K}_\omega =
 \absconv(\tilde{g}_1, \ldots, \tilde{g}_{(1+A)n})$

\begin{lemma}\label{lem:kashin}
There exists a constant $C_1$, depending only on $A$, such that
for each sufficiently large even number $n$ \[\prob_\omega(\is_1)
\geq 1 - 3 \cdot \exp(-n/2),\]
where $\is_1$ is the set of those
$\omega$ for which $\tilde{K}_\omega \cap X \subset C_1
\ball^{An}_2$.
\end{lemma}

\begin{proof}
Let ${\mathcal{U}}$ be the group of unitary operators on
$\R^{(1+A)n}$, with its normalized Haar measure. For $\gau =
(\tilde{g}_i)$, let $U \gau = (U \tilde{g}_i)$. It is well known
(see e.g. \cite[Proposition V.1.1]{MaPis})
that the distributions $(U \gau_\omega)_{U \in {\mathcal{U}},
\omega \in \Omega}$ and $(\gau_\omega)_{\omega \in \Omega}$ are
the same. Define the set $\is_1^\prime$ of all pairs $(U,\omega)$
for which $\tilde{K}_\omega \cap U(X) \subset C_1 \ball^{An}_2$.
Then $\prob_\omega(\is_1) = \prob_{\omega,U}(\is_1^\prime)$. For
any $\omega$, let $\is_{1 \omega}^\prime$ be the set of all $U \in
{\mathcal{U}}$ for which $(\omega, U) \in \is_1^\prime$. It
suffices to show that
\begin{equation}
\label{eq:prob} \prob_\omega \big( \prob_U (\is_{1 \omega}^\prime)
 \geq 1 - 2 \cdot \exp(-n/2) \big) \geq 1 - 2^{-n} .
\end{equation}

Consider the set ${\mathcal{F}}$ of all $\omega$ for which there
exists a subspace $F$ of codimension $n/2$ in $\R^{(1+A)n}$, so
that
$$
F \cap \tilde{K}_\omega \subset F \cap C_1^\prime
\ball^{(1+A)n}_2,
$$
where $C_1^\prime$ is a constant (depending only on $A$). By
\cite[Theorem 2.4]{LPT06}, if $\omega \in {\mathcal{F}}$, then
$\prob_U (\is_{1 \omega}^\prime)
 \geq 1 - 2 \cdot \exp(-n/2)$ if
$C_1 = C_1^\prime ( \kappa A)^{3/2}$, where $\kappa$ is a
universal constant. To prove \eqref{eq:prob}, we need to show that
$\prob_\omega({\mathcal{F}}) \geq 1 - 2^{-n}$.\medskip

To establish the last inequality, consider the (random) operator
$\Gamma_\omega$, mapping $e_i$ ($1 \leq i \leq (1+A)n$) to
$\tilde{g}_{i,\omega}$. It is well known (see \cite[Lemma
2.8]{Sza90}) that there exists an absolute constant $\lambda > 0$
so that
$$
\prob_\omega(\|\Gamma_\omega\| \geq \lambda \sqrt{(1+A)n}) \leq
\exp(-(1+A)n)
$$
for sufficiently large $n$ (here we consider $\Gamma_\omega$ as an
operator $\ell_2^{(1+A)n}\mapsto \ell_2^{(1+A)n}$).
\medskip

On the other hand, by the well-known Kashin decomposition
\cite{Kas77} (see also \cite{Sza78} and \cite[Theorem
6.1]{Pis89}), there exists a subspace $G \subset \R^{(1+A)n}$, of
codimension $n/2$, so that
$$
\sqrt{(1+A)n}\, \ball_1^{(1+A)n} \cap G \subset 20^{2(1+A)}
\ball_2^{(1+A)n}.
$$
In fact, most subspaces of given (proportional) codimension have
this property, but one subspace is enough for us. If $\omega$
satisfies $\|\Gamma_\omega\| \leq \lambda \sqrt{(A+1)n}$, we let
$F = \Gamma_\omega(G)$. Note that $\Gamma_\omega$ maps
$\ball_1^{(1+A)n}$ onto $\tilde{K}_\omega$, hence $F \cap
\tilde{K}_\omega \subset F \cap C_1^\prime \ball^{(1+A)n}_2$ for
$C_1^\prime = \lambda 20^{2(1+A)}$.
\end{proof}

Keeping the notation of Lemma \ref{lem:kashin}, we obtain:

\begin{corollary}\label{cor:d_n^a}
For any $\omega \in \is_1$, $d_n^a(K_\omega) \leq C_1$,
where $C_1$ is the constant from Lemma \ref{lem:kashin}.
\end{corollary}

\begin{proof} Let $\tilde{X}$ be the normed space defined as
$\R^{(1+A)n}$ with the norm whose unit ball is $\ball(\tilde{X}) =
\conv(C_1^{-1} \tilde{K}_\omega \cup \ball(X))$. Clearly,
$\ball(\tilde{X}) \cap X = \ball(X)$, hence the embedding of $X$
into $\tilde{X}$ is isometric.

On the other hand, $d_n(C_1^{-1} K_\omega, \tilde{X}) \leq 1$. In
fact, the space $X^\perp = \ker P_X$ (the orthogonal complement of
$X$ in $\tilde{X}$) is $n$-dimensional. In addition, for any $x
\in C_1^{-1} K_\omega$ there exists $\tilde{x} \in C_1^{-1}
\tilde{K}_\omega \cap P_X^{-1}(x)$. Therefore, $x - \tilde{x} \in
X^\perp$, and $\|\tilde{x}\|_{\tilde{X}} \leq 1$. Thus,
$d_n(C_1^{-1} K_\omega, \tilde{X}) \leq 1$.
\end{proof}

Thus, with overwhelming probability, $d_n^a({K}_\omega) \leq C_1$.
We shall show that, with overwhelming probability,
$d_n(K_\omega,X) \ge 4 n^{\sigma - \gamma}$.
\medskip

The following easy observation provides a useful tool for us. If
$E$ is a subspace of $X$, denote by $P_E$ the orthogonal
projection from $X$ (or $\tilde{X}$) onto $E$. We shall view $E$
as equipped with the norm whose unit ball $\ball(E) =
P_E(\ball(X))$.

\begin{lemma}\label{lem:criterion}
Suppose $S$ is a subset of $X$. Then $d_m(S, X) \geq c$ if and
only if for every $E \subset X$ with $\codim E = m$, we have
$P_E(S) \nsubseteq c \ball(E)$.
\end{lemma}

\begin{proof}
The proof can be viewed as a standard exercise: the orthogonal
complement of $E$ satisfying $P_E(S) \subseteq c \ball(E)$ is a
subspace witnessing $d_m(S, X) \leq c$.\end{proof}

We have to show that, with high probability,
$P_E(\tilde{K}_\omega) \nsubseteq C_2 n^{\sigma - \gamma}
\ball(E)$ holds for any $E$ of dimension $(A-1)n$ and some $C_2$,
when $n$ is large enough. Note that $P_E(\tilde{K}_\omega)$ is the
absolute convex hull of the vectors $g_{E,i} := P_E g_i = P_E
\tilde{g}_i$ ($1 \leq i \leq (1+A)n$), which are independent
$(A-1)n$-standard Gaussians.\medskip

Our next auxiliary result is well known.
For the sake of brevity, set $\vv = \vora(X)$.

\begin{lemma}\label{lem:t-net}
For any $t \in (0,1]$, $\ball(X)$ contains a set $(x_i)_{i=1}^N$,
with $N \leq ((1+2t^{-1}) \vv)^{An}$, so that, for every $x \in
\ball(X)$, there exists $i$ satisfying $\|x - x_i\|_2 \leq t$.
\end{lemma}

\begin{proof}
Suppose $(x_i)_{i=1}^N$ is a maximal subset of $\ball(X)$ with the
property that $\|x_i - x_j\|_2 > t$ whenever $i \neq j$. Consider
$S = \cup_i \{x_i + t/2 \ball^{An}_2\}$ (a disjoint union of $N$
balls). Then $S \subset \ball(X) + t/2 \ball^{An}_2 \subset
(1+t/2) \ball(X)$, hence
$$
N (t/2)^{An} \vol(\ball^{An}_2) = \vol(S) \leq (1+t/2)^{An}
\vol(\ball(X)) \leq (1+t/2)^{An} \vv^{An} \vol(\ball^{An}_2) ,
$$
yielding the desired inequality.
\end{proof}

\begin{corollary}\label{cor:vol_proj}
If $E$ is a subspace of $X$ of dimension $(A-1)n$, then
$\vol(\ball(E)) \leq 3^{An} \vv^{An} \vol(\ball^{(A-1)n}_2)$.
\end{corollary}

\begin{proof}
Suppose $(x_i)_{i=1}^N$ is as in the statement of
Lemma~\ref{lem:t-net}, with $t=1$ (hence $N \leq 3^{An}
\vv^{An}$). Then $\ball(X) \subset \cup_{i=1}^N \{x_i +
\ball^{An}_2\}$, hence
$$
\ball(E) = P_E(\ball(X)) \subset \cup_{i=1}^N \{P_Ex_i +
\ball^{(A-1)n}_2\} .
$$
Therefore, $\vol(\ball(E)) \leq N \vol(\ball^{(A-1)n}_2)$.
\end{proof}

\begin{lemma}\label{lem:K in E}
For any $\lambda > 0$, we have: for any $E \subset X$ of dimension
$(A-1)n$,
$$
\prob \big(P_E(K_\omega) \subset \lambda \ball(E)\big) \leq \Big(
\frac{\vv^\prime}{\sqrt{(A-1)n}} \lambda \Big)^{(A-1)(A+1)n^2} ,
$$
where $\vv^\prime = (3 \vv)^{A/(A-1)} \sqrt{e}$.
\end{lemma}

\begin{proof}
Recall that $P_E(K_\omega)$ is the absolute convex hull of
$(1+A)n$ independent $(A-1)n$-standard Gaussian vectors $g_{E,i}$.
Thus,
$$
\prob \big(P_E(K_\omega) \subset \lambda \ball(E) \big) = \Big(
\prob \big( g \in \lambda \ball(E) \big) \Big)^{(1+A)n} ,
$$
where $g$ is a $(A-1)n$-standard Gaussian vector. By \cite[Fact
1]{MT03},
$$  \eqalign{
\prob(g \in \lambda \ball(E)) & \leq e^{(A-1)n/2}
\vol\big(((A-1)n)^{-1/2} \lambda
\ball(E)\big)/\vol(\ball^{(A-1)n}_2) \cr & \leq \Big(
\frac{e}{(A-1)n} \Big)^{(A-1)n/2} (3\vv)^{An} \lambda^{(A-1)n} . }
$$ Therefore,
$$
\prob \big(P_E(K_\omega) \subset \lambda \ball(E) \big) \leq \Big(
\frac{\vv^\prime}{\sqrt{(A-1)n}} \lambda \Big)^{(A-1)(A+1)n^2} . \, \,
\qedhere
$$
\end{proof}

Denote by ${\mathcal{E}}$ the set of all subspaces of $X$ of
dimension $(A-1)n$, equipped with the distance $\dist(E,F) = \|P_E
- P_F\|_2$. Here, for an operator $T \in B(E)$, we denote by $\|
\cdot \|_2$ its operator norm on $\ell_2^{An}$.

\begin{lemma}\label{lem:nearby}
For any $E, F \in {\mathcal{E}}$, and $x \in X$,
$$
\|P_F x\|_F \leq
 \|P_E x\|_E + (\|P_E x\|_E \sqrt{An} + \|x\|_2) \|P_E - P_F\|_2 .
$$
\end{lemma}

\begin{proof}
For simplicity, let $a = \|P_E x\|_E$, and $b = \|x\|_2$. By the
definition of the norm on $E$, we can write $x = x_1 + x_2$, with
$x_1 \in a \ball(X)$, and $x_2 \in E^\perp$. Recall that
$\ball^{An}_2$ is the maximal volume ellipsoid contained in
$\ball(X)$, hence, by the well known theorem of F.~John (see
\cite[p.~10]{MS86}), $\ball(X) \subset \sqrt{An}\, \ball^{An}_2$.
Therefore, $\|x_2\|_2 \leq \|x_1\|_2 + \|x\|_2 \leq a \sqrt{An} +
b$. We have
$$
P_F x = P_F x_1 + P_F x_2 = P_F x_1 + (P_F - P_E) x_2 .
$$
Thus,
$$
\|P_F x\|_F \leq \|P_F x_1\|_F + \|(P_F - P_E) x_2\|_2 \leq a +
\|P_F - P_E\|_2 \|x_2\|_2 \leq a + \|P_F - P_E\|_2 (a \sqrt{An} +
b) .
$$
\end{proof}

\begin{corollary}\label{cor:nearby}
Suppose $E \in {\mathcal{E}}$ and $\omega$ are such that
$$
P_E(K_\omega) \subset a \ball(E) ,
$$
and
$$
\max_{1 \leq i \leq (A+1)n} \|\tilde{g}_i\|_2 \leq b \sqrt{An} .
$$
Then, for any $F \in {\mathcal{E}}$,
$$
P_F(K_\omega) \subset \big(a + \|P_F - P_E\|_2 (a+b) \sqrt{An}
\big) \ball(F) .
$$
\end{corollary}

\begin{proof}[Proof of Theorem \ref{thm:vol_rat}]
Consider the set $\is_2$ of all $\omega$ for which $\|g_i\|_2 \leq
4 \sqrt{(A-1)n}$ for every $i$. By \cite[Fact 1]{MT03}, if $g$ is
an $An$-standard Gaussian, then
$$
\prob \big( \|g\|_2 > 4 \sqrt{(A-1)n} \big) \leq \big( \sqrt{2}
e^{-4(A-1)/A} \big)^{An} ,
$$
hence
\begin{equation}
\prob(\is_2) \geq 1 - (A+1) n \big( \sqrt{2} e^{-4(A-1)/A}
\big)^{An} \geq 1 - e^{-2(A+1)n}
\label{eq:S_2}
\end{equation}
for $n$ large enough (recall that $A \geq 5$).

We shall prove that, for $n$ large enough, there exists $\omega
\in \is_1 \cap \is_2$, with the property that $P_E(K_\omega)
\nsubseteq C \ball(E)$ for any $E \in {\mathcal{E}}$, where $C = 4
n^{\sigma - \gamma}$ ($\is_1$ is defined as in Lemma \ref{lem:kashin}).

Let $t = (An)^{-1/2}$. By \cite{Sza81} (see also \cite[Proposition 6]{Paj99}),
${\mathcal{E}}$ has a $t$-net ${\mathcal{E}}^\dagger$, of
cardinality not exceeding $(C_3/t)^{(A-1)n^2}$, where $C_3$ is a
universal constant. Suppose $P_E(K_\omega) \subset C \ball(E)$,
for some $E$. Find $F \in {\mathcal{E}}^\dagger$ so that $\|P_E -
P_F\|_2 \leq t$. By Corollary \ref{cor:nearby}, $P_F(K_\omega)
\subset (2C + 4) \ball(F)$.

Denote by $\is_{3,F}$ the set of all $\omega \in \is_2$ for which
$P_F(K_\omega) \subset (2C + 4) \ball(F)$, and let $\is_3 =
\cup_{F \in {\mathcal{E}}^\dagger} \is_{3,F}$. For a given $F$,
Lemma \ref{lem:K in E} yields
$$
\prob(\is_{3,F})
\leq \Big( \frac{\vv^\prime}{\sqrt{(A-1)n}} (2C+4)
\Big)^{(A-1)(A+1)n^2} \leq \Big( \frac{\vv^\prime}{\sqrt{An}} 3C
\Big)^{(A-1)(A+1)n^2}
$$
Thus,
$$  \eqalign{
& \prob(\is_3) \leq |{\mathcal{E}}^\dagger| \Big(
\frac{\vv^\prime}{\sqrt{An}} 3C \Big)^{(A-1)(A+1)n^2} \cr & \leq
\big( C_3 \sqrt{An} \big)^{(A-1)n^2} \Big(
\frac{\vv^\prime}{\sqrt{An}} 3C \Big)^{(A-1)(A+1)n^2} \cr & =
\Big( C_3 (An)^{-A/2} \big( 3 \vv^\prime C \big)^{A+1}
\Big)^{(A-1)n^2} . }
$$
Note that $C^{A+1} = 4^{A+1} n^{(\sigma - \gamma)(A+1)}$, and
$\vv^{\prime (A+1)} \leq n^{\gamma A(A+1)/(A-1)}$. By our choice
of $A$,
$$
\frac{A}{2} > (\sigma - \gamma)(A+1) + \gamma
\frac{A(A+1)}{A-1} .
$$
and therefore, $\prob(\is_3) \leq (C_4 n)^{-C_5 n^2}$, where
$C_4$ and $C_5$ are positive constants.

On the other hand, combining Lemma \ref{lem:kashin} with \eqref{eq:S_2}, we
obtain, for $n$ large enough,
$$
\prob(\is_1 \cap \is_2) \geq 1 - 3 e^{-n/2} - e^{-2(A+1)n} .
$$
Thus, for large $n$, $\prob(\is_3) < \prob(\is_1 \cap \is_2)$.
Thus, there exists $\omega \in \is_1 \cap \is_2$, so that
$P_E(K_\omega) \nsubseteq C \ball(E)$, for any $E$. By
Lemma~\ref{lem:criterion}, we are done.
\end{proof}

To prove Theorem \ref{thm:unif_bded_VR}, we need also the
following lemma.

\begin{lemma}\label{lem:subsp}
Suppose $X$ is an $m$-dimensional space. Then, for any $k \leq m$,
there exists a $k$-dimensional subspace $Y$, so that $\dim Y = k$,
and $\vora(Y) \leq \vora(X)$.
\end{lemma}

\begin{proof}
Denote the norm of $X$ by $\| \cdot \|$. Without loss of
generality, the maximal volume ellipsoid inscribed into $\ball(X)$
is the Euclidean ball. By e.g. \cite[Section 6]{Pis89},
$$
\vora(X) = \int_{{\mathbf S}^{m-1}} \|x\|^{-m} \, d\sigma_{m-1} ,
$$
where $\sigma_{m-1}$ is the uniform probability measure on the
unit sphere ${\mathbf S}^{m-1}$. As explained in e.g.
\cite[1.6]{MS86}, we can write
$$
\vora(X) = \int_{\mathbf{G}} \int_{{\mathbf S}^{k-1}(Y)}
\|x\|^{-m}
 \, d\sigma_{k-1} \, d\mu ,
$$
where $\mu$ is the rotation invariant probability measure on the
Grassman manifold ${\mathbf{G}}$ of $k$-dimensional subspaces $Y
\subset X$, and $\sigma_{k-1}$ is the probability measure on the
unit sphere of $Y$. Clearly, for some $Y \in {\mathbf{G}}$,
$$
\int_{{\mathbf S}^{k-1}(Y)} \|x\|^{-m}  \, d\sigma_{k-1} \leq
\vora(X) .
$$
Then
\[
\vora(Y) = \int_{{\mathbf S}^{k-1}(Y)} \|x\|^{-k}  \,
d\sigma_{k-1}\leq \int_{{\mathbf S}^{k-1}(Y)} \|x\|^{-m}  \,
d\sigma_{k-1} \leq \vora(X).\hskip1cm \qedhere
\]
\end{proof}

\begin{proof}[Proof of Theorem \ref{thm:unif_bded_VR}]
Pick $\sigma \in (\gamma, 1/2)$. As in Theorem \ref{thm:vol_rat},
find a positive integer $A \geq 5$, so that
$$
\frac{A-2}{2(A+1)} \geq \gamma \frac{A}{A-1} + (\sigma - \gamma).
$$
Now we use Lemma \ref{lem:subsp} to obtain a sequence
$\{X_n\}_{n=1}^\infty$ of uniformly complemented subspaces
so that $\dim X_n=A k_n$, where $k_n$ is even,
$\lim_{n\to\infty}k_n = \infty$), and $\vora(X_n) \leq k_n^\gamma$.
Theorem \ref{thm:vol_rat} yields, for $n$ large enough, compact sets
$K_n \subset X_n$, so that $\sup_n d_{k_n}^a(K_n) < \infty$, and
$\lim_n d_{k_n}(K_n,X) = \infty$.
\end{proof}

Can we use the techniques of Theorem \ref{thm:unif_bded_VR} for other spaces?
Below, we outline a possible approach. As in Section \ref{s:abs wi},
we use the notation
$\lambda(F)$ and $\lambda(F,G)$ for absolute and relative projection
constants. On the first step, find (when  possible) a sequence of
uniformly complemented subspaces $X_n\subset X$ such that
$\lambda(X_n)\to\infty$. The second step consists of picking a sequence $\{Y_n\}$
of superspaces $Y_n\supset X_n$ such that $\lim_n \lambda(X_n, Y_n) = \infty$,
and $k_n = \dim(Y_n/X_n)=\dim X_n/2$ (or more generally,
$\lim_n \big(\dim(Y_n/X_n)/\dim X_n\big) = \alpha \in (0,1)$).
The third step proceeds as in the proof Theorem \ref{T:DepSpace} --
namely, by selecting projections $P_n:Y_n\to X_n$ so that
$\lim_n d_{k_n}(P_n(\ball(Y_n)), X_n) = \infty$. Then we would
also have $\lim_n d_{k_n}(P_n(\ball(Y_n)), X) = \infty$
(due to the uniform complementability of $X_n$'s), and $d_{k_n}^a(K_n)\le 1$.
We believe that the possibility of implementing the second step of this program is an
interesting problem, which can find other applications as well:

\begin{problem}\label{P:ProjRel}
Suppose that finite-dimensional spaces $X_n$ are such that
$\lambda(X_n)\to\infty$. Does this imply that there exist
$Y_n\supset X_n$ such that \[\dim(Y_n/X_n)\le\dim X_n/2\quad
\hbox{ and }\quad \lambda(X_n,Y_n)\to\infty?\] The problem is of
interest if we replace $2$ by any positive constant.
\end{problem}

Problem \ref{P:ProjRel} can be considered as a problem on
possibility to generalize the isometric, one-codimensional result
of Davis \cite{Dav}.

\medskip
The possibility of making the third step is still a problem (even
if we assume that Problem \ref{P:ProjRel} has a positive answer):
Can $Y_n$ and $P_n$ be chosen in such a way that $P_n(\ball(Y_n))$
has large $k$-width in $X_n$, where $k=\dim(Y_n/X_n)$?

\begin{remark} There exist non-${\mathcal{L}}_\infty$-spaces for which the scheme above cannot be realized
because they do not contain uniformly complemented
finite-dimensional spaces with growing dimensions. One example of
this type was constructed by Pisier \cite{Pis83} (see \cite{Pis86}
for a simpler version of the construction).
\end{remark}

\section{Ratios of widths to absolute widths}\label{s:ratio}

In this section, we modify Problem \ref{P:Nspace}.

\begin{problem}\label{P:ratio}
(1) Describe the Banach spaces $\my$ which contain compact subsets $K$
so that $\limsup_n d_n(K)/d_n^a(K) = \infty$.

(2) What can be said about the Banach spaces $\my$ satisfying
a stronger property: they contain compact subsets $K$
so that $\liminf_n d_n(K)/d_n^a(K) = \infty$.
\end{problem}

To answer Part (1) of this question, we state:

\begin{proposition}\label{limsup}
Suppose a Banach space $\my$ is such that there exist $\gamma > 0$
and $\sigma \in [0,1/2)$ so that, for infinitely many positive integers $n$,
there exist operators $A_n : \ell_2^n \to \my$ and
$B_n : \my \to \ell_2^n$, so that $B_n A_n = I_{\ell_2^n}$, and
$\|A_n\| \|B_n\| \leq \gamma n^\sigma$. Then $\my$ contains a compact
subset $K$, so that
$$
\limsup d_n(K)/d_n^a(K) = \infty .
$$
\end{proposition}

If $\my$ is $K$-convex, then there exists a sequence of
projections $P_n$ from $\my$ onto subspaces $F_n$, where $\sup_n
\|P_n\| < \infty$, and $d(F_n, \ell_2^n) < 2$ (see \cite{Pis82} or
\cite[Theorem 19.3]{DJT}). Thus, $K$-convex spaces $\my$ satisfy
the conditions of this proposition. By \cite[Example 3.5]{FLM},
Proposition \ref{limsup} is also applicable to $\my = (\oplus_n
\ell_1^n)_{c_0}$, $(\oplus_n \ell_1^n)_\infty$, $c_0(\ell_1)$, or
$\ell_\infty(\ell_1)$.

\begin{proof}
Find a sequence $4 < n(1) < n(2) < \ldots$ so that, for any $j \in \N$,
$n(j+1) > 4 n(j)$, and there exist operators $U_j : \ell_2^{n(j)} \to \my$
and $V_j : \my \to \ell_2^{n(j)}$, so that $\|U_j\| \leq 1$, and
$\|V_j\| \leq \gamma n(j)^\sigma$. Define $m(j) = \lceil n(j)/2 \rceil$
and $k(j) = m(j) - \sum_{i=1}^{j-1} m(i)$ (note that $k(j) \geq 3 m(j)/5$).
Furthermore, set $\alpha_1 = 1$, and $\alpha_{j+1} = \alpha_j/\sqrt{n(j)}$.

Let $id_{12}^{(j)}$ be the formal identity map from $\ell_1^{n(j)}$ to
$\ell_2^{n(j)}$, and set $\tilde{K}_j = id_{12}^{(j)} \ball(\ell_1^{n(j)})$.
By \cite{GG84},
$$d_{k(j)}^a(\tilde{K}_j) \leq c_{k(j)}(id_{12}^{(j)}) < C_1 n(j)^{-1/2}$$
($C_1 > 0$ is an absolute constant).
On the other hand, by \cite[Theorem VI.2.7]{Pin85}, $d_{m(j)}(\tilde{K}_j) > 1/2$.

Let $K_j = \alpha_j A_j(\tilde{K}_j)$.
Then the set $K = \conv(K_1, K_2, \ldots)$ is compact and convex.
We claim that, for any $j$,
$d_{m(j)}(K) \geq \alpha_j \gamma^{-1} n(j)^{-\sigma}/2$,
while $d_{m(j)}^a(K) \leq C_1 \alpha_j n(j)^{-1/2}$.

To estimate $d_{m(j)}(K)$ from below,
note that $V_j(K) \supset \alpha_j^{-1} \tilde{K}_j$. By Lemma \ref{L:ops},
$$
\frac{1}{2} < d_{m(j)}(\tilde{K}_j) \leq \alpha_j^{-1} \|V_j\| d_{m(j)}(K) .
$$
As $\|V_j\| \leq \gamma n(j)^\sigma$, we obtain
$d_{m(j)}(K) \geq \alpha_j \gamma^{-1} n(j)^{-\sigma}/2$.

Next obtain an upper estimate for $d_{m(j)}^a(K)$. Embed $\my$ isometrically
into a $1$-injective Banach space $\my^\prime$ (we can take, for instance,
$\my^\prime = \ell_\infty(I)$). Find $F \subset \my^\prime$ so that
$\dim F \leq k(j)$, and $E(K_j,F) \leq C_1 \alpha_j n(j)^{-1/2}$.
Now let $G = \span[F, \ran V_1, \ldots, \ran V_{j-1}]$.
Clearly, $\dim G \leq k(j) + \sum_{i=1}^{j-1} n(i) \leq m(j)$.
We show that $E(K, G) \leq C_1 \alpha_j n(j)^{-1/2}$. By convexity, it suffices to
establish the inequality $E(x, G) \leq C_1 \alpha_j n(j)^{-1/2}$ for $x \in K_s$,
for $s \in \N$. For $s < j$, we have $x \in G$, hence $E(x,G) = 0$. For
$s = j$, $E(x, G) \leq E(x, F) < C_1 \alpha_j n(j)^{-1/2}$, by our
choice of $F$. For $s > j$,
$$
E(x,G) \leq \|x\| \leq \alpha_s \leq \alpha_{j+1} = \alpha_j n(j)^{-1/2} .
$$

Taken together, the results above yield
$d_{m(j)}(K)/d_{m(j)}^a(K) \geq \beta m(j)^{1/2-\sigma}$,
where $\beta$ is a constant.
\end{proof}

In \cite{Ost10}, a special case of the previous proposition was
established: it was proved that $\ell_2$ contains an infinite
dimensional compact $K$ for which
$\limsup_{n\to\infty}d_n(K)/d_n^a(K)=\infty$. This result leads to
the following question \cite[Problem 4.2]{Ost10}: Does there exist
an infinite-dimensional compact $K$ in some Banach space $\my$
such that \[\lim_{n\to\infty}d_n(K)/d_n^a(K)=\infty?\] Below, we
provide a positive answer.

\begin{proposition}\label{P:limit}
\begin{enumerate}
\item
Suppose $1 < p \leq 2$, and $\alpha \in (0,1/q)$, where $1/p + 1/q = 1$.
Then there exists an operator $u_p : \ell_1 \to \ell_p$, so that, for every $n$,
$$
d_n^a(u) \leq c_n(u_p) \leq \beta_{p\alpha} (1 + \log n) n^{-1/q}
\textrm{   and  }
d_n(u_p) \geq \gamma_{p\alpha} n^{-\alpha} .
$$
\item
Suppose $2 < p < \infty$, and $\alpha \in (0,1/p)$. Then there exists an operator
$u_p : \ell_1 \to \ell_p$, so that, for every $n$,
$$
d_n^a(u) \leq c_n(u_p) \leq \beta_{p\alpha} (1 + \log n) n^{-1/2}
\textrm{   and  }
d_n(u_p) \geq \gamma_{p\alpha} n^{1/p-1/2-\alpha} .
$$
\end{enumerate}
Here $\beta_{p\alpha}$ and $\gamma_{p\alpha}$ depend on $p$ and $\alpha$ only.
\end{proposition}

\begin{proof}
By Lemma \ref{L:AbsGel}, $d_n^a(u) \leq c_n(u)$ for any $n$, and
any operator $u$.

Throughout the proof, we denote by $(e_j^{(p)})_{j \in \N}$ the canonical
basis in $\ell_p$. The projection onto the first $N$ elements of this basis
is denoted by $P_N^{(p)}$. For $p \leq q$, $id_{pq}$ ($id_{pq}^N$) stands
for the formal identity from $\ell_p$ to $\ell_q$ (resp. from $\ell_p^N$
to $\ell_q^N$). We identify the range of $P_N^{(p)}$ with $\ell_p^N$.

In both (1) and (2), we consider a diagonal operator $u_p$, taking
$e_j^{(1)}$ to $j^{-\alpha} e_j^{(p)}$. We make repeated use of the
following formula: if $v = \diag(a_j)_{j=1}^\infty$ is a diagonal operator
from $\ell_1$ to $\ell_2$, then, by \cite[Theorem VI.2.7 on p. 207]{Pin85},
\begin{equation}
d_n(u) = \sup_{r > n} \sqrt{\frac{r-n}{\sum_{j=1}^r u_j^{-2}} } .
\label{eq:diag}
\end{equation}

(1) $1 < p \leq 2$.
To estimate $d_n(u_p)$, note that $id_{p2} u_p = u_2$, hence
$d_n(u_p) \geq d_n(u_2)$. By \eqref{eq:diag},
$d_n(u_2) \geq \gamma_\alpha n^{-\alpha}$.
Now let $N = \lceil n^{1/(\alpha q)} \rceil$.
By \cite{GG84},
$$
c_n(id_{1p}^N) \leq \frac{c_p}{\alpha q} \big( 1 + \log n \big)^{1/q} n^{-1/q} ,
$$
for some universal constant $c_p > 1$. Thus, there exists a subspace
$F \subset \span[e_j^{(1)} : 1 \leq j \leq N]$, so that
$$
\|id_{1p}|_F\| \leq \frac{c_p}{\alpha q} \big( 1 + \log n \big)^{1/q} n^{-1/q} .
$$
Denote by $v_p$ the diagonal
operator on $\ell_p^N$, mapping $e_j^{(p)}$ to $j^{-\alpha} e_j^{(p)}$, and
note that $u_p = v_p id_{1p}$. Therefore,
$$
\|u_p|_F\| \leq \frac{c_p}{\alpha q} \big( 1 + \log n \big)^{1/q} n^{-1/q} .
$$
Now let $G = \span[F, e_{N+1}^{(1)}, e_{N+2}^{(1)}, \ldots]$. Then
$\dim \ell_1/G \leq n$, and, by our choice of $N$,
$$
c_n(u_p) \leq \|u_p|_G\| \leq \frac{c_p}{\alpha q} \big( 1 + \log n \big)^{1/q} n^{-1/q} .
$$
As $c_n(u_p) \leq \|u_p|_G\|$, we are done.

(2) $2 \leq p < \infty$.
Note that $u_p = id_{2p} u_2$, and $id_{2p}$ is contractive.
Using the estimates for $c_n(u_2)$ obtained in Part (1), we get:
$$
c_n(u_p) \leq \|id_{2p}\| c_n(u_2) \leq
 \beta_{2\alpha} \big( 1 + \log n \big)^{1/2} n^{-1/2} .
$$
On the other hand, $d_n(u_p) \geq d_n(u_p P_{2n}^{(1)})$.
By \eqref{eq:diag}, $d_n(u_2 P_{2n}^{(1)}) \geq 2\gamma_\alpha n^{-1/\alpha}$,
for some constant $\gamma_\alpha$. Furthermore,
$(id_{2p}^{2n})^{-1} u_p P_{2n}^{(1)} = u_2 P_{2n}^{(1)}$,
hence
$$
d_n(u_p P_{2n}^{(1)}) \geq \|(id_{2p}^{2n})^{-1}\|^{-1} d_n(u_2 P_{2n}^{(1)}) \geq
(2n)^{-(1/2-1/p)} \cdot 2 \gamma_\alpha n^{-1/\alpha} \geq
\gamma_\alpha n^{1/p - 1/2 - \alpha} . \, \, \qedhere
$$
\end{proof}

\begin{problem}\label{P:to infty}
Which Banach spaces $\my$ contain a compact $K$ with the property that
$$
\lim \frac{d_n(K)}{d_n^a(K)} = \infty ?
$$
\end{problem}

By Proposition \ref{P:limit}, the answer is affirmative if $\my$ contains
a complemented copy of $\ell_p$, for some $p \in (1,\infty)$.
This occurs, for instance, for $\my = L_p(\mu)$.
Large classes of rearrangement invariant function spaces contain
complemented copies of $\ell_2$, see e.g.~\cite[Theorem
2.b.4]{LT79}.

\section{Restricted widths}\label{s:restrict}

The following problem was raised in \cite{Ost10}.

\begin{problem}[\cite{Ost10}]\label{P:RevKKM} Characterize compacts $K$ for which the
absolute widths do not differ much from their widths in
$\overline{\span[K]}$.
\end{problem}

The importance of this problem is illustrated by
Lemma \ref{L:CompComp} below.

It is worth mentioning that any Banach space $\my$ contains a compact
$K$ whose widths in $\overline{\span[K]}$ are the same as
the absolute widths. To construct an example, we use a technique
of Tikhomirov \cite{Tik60}.
Let $\{Z_n\}$ be a family of subspaces in a Banach space $\my$
satisfying $\dim Z_n=n$ and $Z_n\subset Z_{n+1}$, let $\ball_n$ be
their unit balls and let $\{t_n\}$ be a decreasing sequence of
positive numbers with $\lim_{n\to\infty}t_n=0$. Consider the
compact
$$
K=\overline{\conv \left(\cup_{n=1}^\infty t_n \ball_n\right)}. $$ 
Then
$d_n(K,\mx)=t_{n+1}$ for each $n\in\mathbb{N}$ and each Banach space
$\mx$ containing $\overline{\span[K]}$ as a subspace.
\medskip
The reasons: (1) Estimate from above: $K\subset Z_n+t_{n+1}\ball(\mx)$.
(2) Estimate from below: $K\supset t_{n+1}\ball_{n+1}$ and the result
of \cite{KKM48} saying that the maximal distance from a unit ball
of an $(n+1)$-dimensional subspace to an $n$-dimensional subspace
is equal to $1$.

There are other classes of $K$'s for which $d_n(K) = d_n^a(K)$ holds.
Suppose $1\le q\le p\le \infty$.
In \cite{Oik95} it was shown that the natural image of $\ball(\ell_p^m)$ in $\ell_q^m$
satisfies this. Furthermore \cite{Koc90}, $d_n(u)=d_n^a(u)$ if $u : \ell_p^m \to \ell_q^m$
is a diagonal map. Another example of a set $K$ with $d_n(K)=d_n^a(K)$ is provided below.

\begin{proposition}\label{P:uncond}
Suppose $F$ is an $m$-dimensional space with a $1$-unconditional basis $(f_i)_{i=1}^m$,
and $id : \ell_\infty^m \to F$ is the formal identity map, taking $\delta_i$ to $f_i$
for every $i$ (here, $(\delta_i)_{i=1}^m$ denotes the canonical basis for $\ell_\infty^m$).
Then $d_n(id) = d_n^a(id)$ for any $n$.
\end{proposition}

\begin{proof}
If $n \geq m$, we have $d_n(id) = d_n^a(id) = 0$. Now consider $n
\in \{1, \ldots, m-1\}$. Relabeling if necessary, we can assume
that $C = \|\sum_{i=1}^{m-n} f_i\|_F \leq \|\sum_{i \in
\mathcal{F}} f_i\|_F$ whenever $|{\mathcal{F}}| = m-n$. We claim
that $d_n(id) = d_n^a(id) = C$. First take $G = \span[f_i : m-n <
i \leq m]$, and let $q_G : F \to F/G$ be the quotient map. By the
$1$-unconditionality of $(f_i)$, $d_n(id) \leq \|q_G \circ id\| =
C$. For the opposite inequality, we apply \cite[Lemma 4]{Oik95} in
the situation where $V$ is the unit cube. A direct calculation
shows that $d_n^a(id) \geq C$.
\end{proof}

\section{Widths of images of compacts under compact operators}\label{s:image}

The purpose of this section is to make some comments on the
following intriguing problem

\begin{problem}\label{P:omalComp} Let $K$ be a compact in a Banach space $\mx$ and
$T:\mx\to\my$ be a compact operator. Does it follow that
$d_n(TK)=o(d_n(K))$?
\end{problem}

Set $\hat{d}_n(K) = d_n(K, \overline{\span[K]}$. \cite[Lemma 6.1]{OS09} states:

\begin{lemma}[\cite{OS09}]\label{L:CompComp}
Let $\mx$ and $\my$ be Banach spaces, $K$ be a compact set in
$\mx$ and $T:\mx\to\my$ be a compact operator. Then
$\hat{d}_n(TK)/\hat{d}_n(K)\to 0$ as $n\to\infty$.
\end{lemma}

For Hilbert spaces $\hat{d}_n(K)=d_n(K)$ and so the result of
Lemma \ref{L:CompComp} remains true if we replace $\hat{d}_n$ by
$d_n$. Problem \ref{P:omalComp} asks whether one can generalize
this result to the Banach space case. Of course, Problem
\ref{P:omalComp} would be solved if one would prove that $\hat
d_n(K)\le Cd_n(K)$ for some absolute constant $C$. However, as we
know, for example, from Theorem \ref{T:DepSpace} this turned out
not to be the case.\medskip

If a compact $K$ is such that $\{d_n(K)\}$ decreases more slowly
than a geometric progression, then $d_n(TK) =o(d_n(K))$. More
precisely:

\begin{proposition}\label{P:slow_decay}
Suppose a compact $K \subset \mx$ and $C\in(1,\infty)$
have the following property:
for any $k\in\mathbb{N}$ there exists $N\in\mathbb{N}$ such that
$d_n(K)/d_{n+k}(K)<C$ for each $n\ge N$. Then $d_n(TK) =o(d_n(K))$
for each compact operator $T:\mx\to\my$.
\end{proposition}
\begin{proof} It suffices to show that for each $\delta>0$ there exists
$M \in \mathbb{N}$ such that $d_m(TK)\le C\delta d_m(K)$ for each
$m\ge M$. To show this we observe that for each $\delta>0$ there
exists $k \in \mathbb{N}$ and a $k$-dimensional subspace
$\my_k\subset \my$ such that
\begin{equation}\label{E:myk}
T \ball(\mx)\subset\my_k+\delta \ball(\my).
\end{equation}
By the assumption there exists $N$ such that $d_{n}(K)<
Cd_{n+k}(K)$ for each $n\ge N$. Let $M\ge N+k$ and $m\ge M$. Then
$d_{m-k}(K)< Cd_m(K)$ and therefore there is an
$(m-k)$-dimensional subspace $\mx_{m-k}\subset\mx$ such that
\[
K\subset \mx_{m-k}+Cd_m(K)\ball(\mx).
\]
Combining with \eqref{E:myk} we get
\[
TK\subset T\mx_{m-k}+Cd_m(K)T\ball(\mx)
 \subset T\mx_{m-k}+\my_k+C\delta d_m(K)\ball(\my).
\]
The subspace $T\mx_{m-k}+\my_k$ is at most $m$-dimensional,
therefore $d_m(TK)\le C\delta d_m(K)$.
\end{proof}
\begin{proposition}\label{P:o_comp}
Let $K$ be a compact subset of a Banach space $X$, and $T : X \to
Y$ be a compact operator. Let $\phi : \N \to \N$ be a function,
satisfying $\lim_n (\phi(n) - n) = + \infty$. Then
$d_{\phi(n)}(TK)=o(d_n(K))$.
\end{proposition}

\begin{lemma}\label{L:in_separable}
Suppose $K$ is a compact subset of a Banach space $\mx$, and
$(\delta_n)$ is a sequence of positive numbers. Then $\mx$ contains
a separable subspace $\tilde{\mx}$ such that, for every $n \in \N$,
$d_n(K, \tilde{\mx}) \leq (1 + \delta_n) d_n(K,\mx)$.
\end{lemma}

\begin{proof}
For each $n \in \N$ find an $n$-dimensional subspace $Z_n \subset
\mx$ such that $E(K,Z_n) \leq (1 + \delta_n) d_n(K,\mx)$. We can
take $\tilde{\mx}$ to be the closure of $\span[K, Z_1, Z_2,
\ldots]$ in $\mx$.
\end{proof}

\begin{proof}[Proof of Proposition~\ref{P:o_comp}]
By Lemma~\ref{L:in_separable}, we can assume that $\mx$ is
separable. Furthermore, we assume that $d_n(K) > 0$ for every $n$
(otherwise, the conclusion of the proposition is immediate). Let
$(x_i)_{i=1}^\infty$ be a countable dense subset of the unit
sphere of $X$. For $n \in \N$, let $\psi(n)$ be the smallest
positive integer $m$ with the property that $\phi(k) - k \geq n$
for any $k \geq m$. Let $\tilde{K}$ be the closed convex hull of
the union of $K$ and the sequence $(d_{\psi(i)}(K) x_i)$. Then
$d_{\phi(n)}(\tilde{K}) \leq d_n(K)$. Indeed, fix $c > 1$, and
find an $n$-dimensional subspace $Z$ in $\mx$, such that $E(K,Z) <
c d_n(K)$. Let $\tilde{Z}$ be the linear span of $Z$, and of $x_1,
\ldots, x_{\phi(n) - n}$. Then $\dim \tilde{Z} \leq \phi(n)$, and
$E(\tilde{K}, \tilde{Z}) \leq c d_n(K)$. As $c > 1$ is arbitrary,
we conclude that $d_{\phi(n)}(\tilde{K}) \leq d_n(K)$. We conclude
the proof by applying Lemma~\ref{L:CompComp} to $\tilde{K}$.
\end{proof}

It may be tempting to approach Problem \ref{P:omalComp} by
fixing $C_1 > C > 1$, finding subspaces $Z_n \hookrightarrow \mx$
such that $E(K, Z_n) \leq C d_n(K)$ and $\dim Z_n = n$, and then
considering $\tilde{K} = \cap_n (Z_n + C_1 d_n(K) \ball(\mx))$ as
a subset of $\tilde{\mx} = \overline{\span[Z_n : n \in \N]} \subset \mx$.
Then $K \subset \tilde{K}$, and $d_n(\tilde{K}, \tilde{\mx}) \leq C_1 d_n(K,\mx)$.
If we had $\tilde{\mx} = \overline{\span[\tilde{K}]}$, we would then
use Lemma \ref{L:CompComp} to conclude that
$$
\frac{d_n(T \tilde{K})}{\hat{d}_n(K)} \leq
\frac{d_n(TK)}{\hat{d}_n(K)} {\underset{n \to \infty}{\longrightarrow}} 0
$$
However, the above construction may lead to
$\overline{\span[\tilde{K}]}$ being a strict subset of
$\tilde{\mx}$, as the following example shows.
Let $\mx = \ell_2$, 
and take $K$ to be the set of all
$(x_i) \in \ell_2$ s.t. $x_1 = 0$, and
$|x_2|^2 + \sum_{i=3}^\infty 4^{3-i} |x_i|^2 \leq 1$.
By \cite{Pie87}, $d_1(K) = 1$, and $d_n(K) = 2^{2-n}$ for $n \geq 2$.
Take $Z_1 = \span[e_1]$, and $Z_n = \span[e_3, \ldots, e_{n+1}]$
for $n \geq 2$. Then $E(K,Z_n) = d_n(K)$ for any $n$.
However, $Z_1 \cap \span[\tilde{K}] = \{0\}$. Indeed,
denote by $P$ the orthogonal projection onto $\span[e_1]$.
Then, for $n \geq 2$ and $x \in Z_n + C_1 d_n(K) \ball(\mx)$,
$\|P x\| \leq 2^{n-2} C_1$. Consequently, for $x \in \tilde{K}$,
we have $P x = 0$. In other words, $\tilde{K} \subset Z_1^\perp$.

\renewcommand{\refname}{

\end{large}

\end{document}